\documentclass[11pt,a4paper]{amsart}

\usepackage{amssymb}
\usepackage{amsmath}
\usepackage{amsfonts}
\usepackage{url}

\usepackage[english]{babel}
\usepackage[T1]{fontenc}
\usepackage[latin1]{inputenc}
\usepackage{amsthm}

\theoremstyle{plain}
\newtheorem{theorem}{Theorem}

\newtheorem{corollary}[theorem]{Corollary}

\theoremstyle{remark}

\newtheorem{algorithm}{Algorithm}

\newtheorem{remark}[theorem]{Remark}

\newcommand{\e}{{\mathrm e}}

\title[Performance guarantees submodular maximization]{New performance guarantees for the greedy maximization of submodular set functions} 

\author{Jussi Laitila}
\author{Atte Moilanen}
\address{Department of Biosciences, P.O. Box 65, FI-00014 University of Helsinki, Finland}
\email{jussi.laitila@helsinki.fi}
\email{atte.moilanen@helsinki.fi}
\subjclass[2010]{90C59, 90C30, 68W25}
\keywords{approximation, cardinality, convex optimization, greedy algorithm, maximization, steepest ascent}
\thanks{J.L. and A.M were supported by the ERC-StG grant 260393. A.M. was supported by the Academy of Finland Centre of Excellence programme 2012 - 2017, grant 250444, and the Finnish Natural Heritage Services (Mets\"ahallitus).}

\date{\today}

\begin{document}

\maketitle

\begin{abstract}
We present new tight performance guarantees for the greedy maximization of nondecreasing submodular set functions. 
Our main result first provides a performance guarantee in terms of the overlap of the optimal and greedy solutions. 
As a consequence we improve performance guarantees of Nemhauser, Wolsey and Fisher (1978) and Conforti and Cornu\'ejols (1984) for maximization over subsets, which are at least half the size of the problem domain. As a further application, we obtain a new tight performance guarantee in terms of the cardinality of the problem domain. 
\keywords{Approximation \and Cardinality \and Convex optimization \and Greedy algorithm  \and Maximization  \and Steepest ascent}
\end{abstract}

\section{Introduction}




Let $X$ be a finite set, $X=\{x_1,\dots, x_n\}$, and let $T$ be an integer such that $0 < T\le n$. We consider the cardinality-constrained maximization 
problem
\begin{align}\label{prob1}
\max\{f(S)\colon |S|=T, S \subset X\},
\end{align}
where $f\colon 2^X\to \mathbb R_+$ is a submodular set function.  
Recall that $f$ is submodular if 
\begin{align}\label{subadd}
f(S)+f(R)\ge f(S\cup R) + f(S\cap R)
\end{align}
for all $S,R\subset X$; 
see, e.g., \cite{NWF78}.
We further assume that $f$ is nondecreasing; $f(S) \le f(R)$ for all $S\subset R$, and, without loss of generality, that $f(\emptyset)=0$. 
We consider the following well-known greedy algorithm for solving problem \eqref{prob1}: 

\begin{algorithm}\label{alga}
\
\begin{description}
\item Step 0: Set $S_0=\emptyset$. Go to Step 1.
\item Step $t$ ($1\le t \le T$): Select any $x_t\in S_{t-1}$ such that
\begin{align*}
f(S_{t-1}\cup \{x_t\})=\max\{f(S_{t-1}\cup \{x\})\colon x\in X\setminus S_{t-i}\}.
\end{align*}
Set $S_t=S_{t-1}\cup \{x_t\}$. Go to step $t+1$.
\item Step $T+1$: Set $S_{gr} = S_{T}$. Stop.
\end{description}
\end{algorithm}

Algorithm \ref{alga} 
has been extensively studied in the literature. By the Rado-Edmonds theorem (\cite{E71} or \cite{CC84}), it finds an optimal solution when $f$ is an additive set function, i.e., when \eqref{subadd} holds with an equality for all $S,R\subset X$. Nemhauser, Wolsey and Fisher \cite{NWF78} (see also \cite{CFN77}, \cite{NW78}) gave the following performance guarantee for 
Algorithm \ref{alga} 
for nonadditive functions $f$: 
\begin{align}\label{NWF}
\frac{f(S_{gr})}{f(S_{opt})}\ge 1-\left(1-\frac{1}{T}\right)^{T}=:G_{NWF}(T),
\end{align}
where $S_{opt}$ is an optimal solution to problem \eqref{prob1}. 
Conforti and Cornu\'ejols \cite{CC84} improved \eqref{NWF} to 
\begin{align}\label{CC}
\frac{f(S_{gr})}{f(S_{opt})}\ge \frac{1}{\alpha}\left(1-\left(1-\frac{\alpha}{T}\right)^{T}\right) =: G_{CC}(T,\alpha),
\end{align}
for $\alpha\in (0,1]$, where $\alpha\in [0,1]$ is the total curvature
\begin{align*}
\alpha=\max\left\{1-\frac{f(X)-f(X\setminus\{x\})}{f(\{x\})-f(\emptyset)}\colon x\in X, f(\{x\})\ne f(\emptyset)\right\}.
\end{align*}
It is known that $\alpha\in (0,1]$ if and only if $f$ is nonadditive \cite{CC84}. 
Clearly, $G_{NWT}(T)=G_{CC}(T,1)$ and since 
$G_{CC}(T,\alpha)\to 1$ as $\alpha \to 0^+$, \eqref{CC} can be viewed as a generalization of the Rado-Edmonds theorem.
The above performance guarantees further satisfy the estimates
\begin{align*}
G_{CC}(T,\alpha)\ge \max\left\{G_{NWF}(T), \frac{1-\e^{-\alpha}}{\alpha}\right\} \ge 1-\e^{-1},
\end{align*}
for all $\alpha$ and $T$.
The guarantees \eqref{NWF} and \eqref{CC} are tight for suitable choices of parameters $T$ and $\alpha$. For example, for all $\alpha\in (0,1]$ and $T\ge 1$ there is a problem of the type \eqref{prob1} and the corresponding greedy solution $S_{gr}$ such that $f(S_{gr})=G_{CC}(T,\alpha)f(S_{opt})$ \cite{CC84}.

Submodular optimization has played a central role in operations research and combinatorial optimization \cite{GS07}. By now it has applications in various fields, including 
computer science \cite{KG14}, economics \cite{T98} and, more recently, ecology (\cite{M07}, \cite{GKKCM11}, \cite{BS12}).
Problem \eqref{prob1} and the above performance guarantees have been extended to various other settings and problem structures, related to, for example, matroid (\cite{FNW78}, \cite{CC84}) and knapsack (\cite{S04}, \cite{KST13}) constraints, continuous algorithms (\cite{V10}, \cite{CCPV11}), nonmonotone functions \cite{FMV11}, nonsubmodular functions \cite{WMWP} and supermodular minimization (\cite{I01}, \cite{IL06}).

To authors' knowledge, previously presented performance guarantees either do not depend on $T$ or $n$, or, like \eqref{NWF} and \eqref{CC}, they are decreasing in $T$. However, when $T=n$, it is clear that $S_{opt}=S_{gr}$, so the greedy algorithm returns the optimal solution. This suggests that any performance guarantee should in fact be 
improving when $T$ approaches and is close enough to $n$. We show that this is indeed the case. 
More generally, we show that increasing degree of overlap $m=|S_{opt}\cap S_{gr}|$ between the sets $S_{opt}$ and $S_{gr}$ improves the performance guarantees. While in applications the overlap $m$ may not be known, we can give this quantity a useful lower bound. In fact, when $T > n/2$, we have $m \ge 2T-n>0$. Our results thus have particular relevance for optimization problems where the maximum is sought over subsets of cardinality larger than $n/2$. 

Let 
\begin{align*}
G(T,\alpha,m)=\frac{1}{\alpha}\left(1-\left(1-\frac{\alpha m}{T}\right)\left(1-\frac{\alpha}{T}\right)^{T-m}\right)
\end{align*}
and $\widetilde{G}(T,\alpha,n)=G(T,\alpha,\max\{0, 2T-n\})$. Our main result is the following.

\begin{theorem}\label{thm1} Let $\alpha \in (0,1]$, let $1\le T \le n$ and let $S_{opt}$ and $S_{gr}$ be an optimal, repectively a greedy, solution to problem \eqref{prob1} and let $m=|S_{opt} \cap S_{gr}|$. Then
\begin{align}\label{thm1eq1}
\frac{f(S_{gr})}{f(S_{opt})}\ge G(T,\alpha,m)\ge \widetilde{G}(T,\alpha,n).
\end{align}
Moreover, these bounds are tight in the following sense: for every $\alpha\in (0,1]$ and numbers $n$ and $T$ such that $1\le T\le n$, there is a problem of the type \eqref{prob1} and its greedy solution $S_{gr}$ such that  $\max\{0,2T-n \}=|S_{opt} \cap S_{gr}|$ and
\begin{align*}
\frac{f(S_{gr})}{f(S_{opt})} = \widetilde{G}(T,\alpha,n).
\end{align*}
\end{theorem}

We postpone the proof of Theorem \ref{thm1} to Section \ref{secproof}.

\begin{remark}
Theorem \ref{thm1} strictly improves \eqref{CC} and provides further examples of cases where the performance guarantee equals one, i.e., generalizations of the Rado-Edmonds theorem. Indeed, 
for all $T$ and $n$ such that $T > n/2$, we have 
\begin{align*}
\widetilde G(T,\alpha,n) > G_{CC}(T,\alpha).
\end{align*}
For $T=n$, we get $\widetilde G(n,\alpha,n)=1$.
Note that, by \eqref{CC}, $\lim_{\alpha\to 0^+}\widetilde G(T,\alpha,n)=1$. 
Moreover, in the case $m=T$, we again get $G(T,\alpha,T)=1$. 
\end{remark}

Using Theorem \ref{thm1}, one can derive other new performance guarantees for the greedy algorithm. As an example of independent interest, we present the following performance guarantee in terms of $n$ only.


\begin{corollary}\label{thm3} Let $\alpha \in (0,1]$, $1\le T \le n$, and let $S_{opt}$ and $S_{gr}$ be an optimal, repectively a greedy, solution to problem \eqref{prob1}. Then
\begin{align}\label{eqcor}
\frac{f(S_{gr})}{f(S_{opt})}\ge \frac{1}{\alpha}\left(1-\left(1-\frac{\alpha}{\left\lfloor\frac{n}{2}\right\rfloor}\right)^{\left\lfloor \frac{n}{2}\right\rfloor}\right)\ge 
\frac{1}{\alpha}\left(1-\left(1-\frac{2\alpha}{n}\right)^{n/2}\right),
\end{align}
where $\lfloor x \rfloor$ denotes the largest integer not greater than $x$.
The left-hand estimate is tight in the following sense: for every $\alpha\in (0,1]$ and $n\ge 2$, there is a problem of the type \eqref{prob1} and its greedy solution $S_{gr}$ such that  
\begin{align*}
\frac{f(S_{gr})}{f(S_{opt})} = \frac{1}{\alpha}\left(1-\left(1-\frac{\alpha}{\left\lfloor\frac{n}{2}\right\rfloor}\right)^{\left\lfloor \frac{n}{2}\right\rfloor}\right).
\end{align*}
\end{corollary}

\begin{proof}
If $n$ is an odd integer, it is easy to check that the minimum of $\widetilde G(T,\alpha,n)$ over all integers $T$ with $0\le T \le n$ is $\widetilde G((n-1)/2,\alpha,n)$. Moreover, when treated as a continuous function of $T$, $\widetilde G(T,\alpha,n)$ attains its minimum at $T=n/2$. Together with Theorem \ref{thm1} this yields \eqref{eqcor}. Tightness of the left-hand inequality in \eqref{eqcor} follows from 
Theorem \ref{thm1} with the choice $T=\left\lfloor\frac{n}{2}\right\rfloor$.
\end{proof}


\section{Proof of Theorem \ref{thm1}}\label{secproof}
In this section we present a proof of Theorem \ref{thm1}. We first prove \eqref{thm1eq1}. Note that the right-hand inequality in \eqref{thm1eq1} follows directly from $m=|S_{opt}\cap S_{gr}|\ge \max\{0, 2T-n\}$ and the fact that $G(T,\alpha,m)$ is increasing in $m$. 

We next prove the left-hand inequality in \eqref{thm1eq1}.
We may assume that $0< m < T$. Indeed, if $m=T$, then $S_{gr} = S_{opt}$ and the claim is trivial. If $m=0$, the claim follows from \eqref{CC}. 

Let $S_0=\emptyset$ and $S_t=\{y_{1},\dots,y_{t}\}\subset X$ be the successive sets chosen by the greedy algorithm for $t=1,\dots,T$, so that $S_0\subset S_1\subset\dots\subset S_T$. 
Let 
\begin{align*}
a_t=\frac{f(S_t)-f(S_{t-1})}{f(S_{opt})},
\end{align*}
for $t=1,\dots,T$.
Because $f$ is nondecreasing, each $a_t$ is nonnegative and 
\begin{align*}
\frac{f(S_{gr})}{f(S_{opt})}=\sum_{i=1}^T a_i.
\end{align*}
Let $J=S_{gr} \cap S_{opt}$. 
Let $1\le j_1\le\dots\le j_m\le T$ denote the indices for which $J=\{y_{j_1},\dots y_{j_m}\}$.
Denote $j_0=0$ and $j_{m+1}=T$.
By Lemma 5.1 of \cite{CC84}, we obtain the $T$ inequalities
\begin{align*}
1 \le \alpha \sum_{\{i\colon y_i\in S_{t-1} \setminus S_{opt}\}}a_i+ \sum_{\{i\colon y_i\in S_{t-1}\cap S_{opt}\}}a_i + (T-|S_{t-1}\cap S_{opt}|)a_{t},
\end{align*}
for $t=1,\dots,T$.  Consequently, 
\begin{align*}
\frac{f(S_{gr})}{f(S_{opt})} \ge B(J'), 
\end{align*}
where
$J' = \{j_1,\dots,j_m\}$ and, for $U\subset \{1,\dots,n\}$, $B(U)$ denotes
the minimum of the linear program
\begin{align}\label{eq:min1}
\textrm{minimize}\quad&\sum_{i=1}^T b_i\\
\textrm{s.t.}\quad&
\alpha \sum_{i\in V_{t-1}\setminus U}b_i+ \sum_{i\in U\cap V_{t-1}}b_i + (T-|U\cap V_{t-1}|)b_{t} \ge 1,\nonumber\\
&b_t\ge 0,\nonumber
\end{align}
for $t=1,\dots,T$, where $V_t=\{1,\dots,t\}$. We next apply the proof of \cite[Lemma 5.2]{CC84}, which implies the following two facts:
\begin{itemize}
\item[(i)] If $T\notin U$, then $B(U)\ge B(\{T-|U\cap V_{T-1}|,\dots,T\})$,
\item[(ii)] $B(\{T-l,\dots,T\})\ge B(\{T-l+1,\dots,T\})$, for all $1\le l\le T-1$.
\end{itemize}
In particular, if $j_m < T$, then $B(J')\ge B(\{T-m,\dots,T\})\ge B(\{T-m+1,\dots,T\})$. Moreover, if $j_m=T$, then $B(J')=B(J'\setminus \{T\})$, so that using (i), $B(J')\ge B(\{T-m+1,\dots,T\})$. 
Consequently,
\begin{align*}
\frac{f(S_{gr})}{f(S_{opt})} \ge \sum_{i=1}^T b^*_i, 
\end{align*}
where $b^*=(b^*_1,\dots,b^*_T)$ is an optimal solution to the problem \eqref{eq:min1} with $U=\{T-m+1,\dots,T\}$.
By the weak duality theorem, we get that
\begin{align*}
\frac{f(S_{gr})}{f(S_{opt})}\ge \sum_{i=1}^T c^*_i,
\end{align*}
where $c^*=(c^*_1,\dots,c^*_T)$ is an optimal solution to the dual problem
\begin{align}\label{eq:dualprob1}
\textrm{maximize}\quad&\sum_{i=1}^T c_i\\
\textrm{s.t.}\quad
& Tc_t+\alpha\sum_{i=t+1}^T c_i\le 1, && 1\le t \le T-m\nonumber\\
& (2T-m+1-t)c_t+\sum_{i=t+1}^T c_i\le 1, && T-m+1 \le t \le T\nonumber\\
&c_i\ge 0, && i=1,\dots,T.\nonumber
\end{align}
Define the vector $c=(c_1,\dots,c_T)$ by
\[ 
c_t = \begin{cases} 
      \frac{1}{T}\left(1-\frac{\alpha m}{T} \right)\left(1-\frac{\alpha}{T}\right)^{T-m-t}, & 1\le t \le T-m,\\
      \frac{T-m}{(2T-m+1-t)(2T-m-t)}, & T-m+1 \le t \le T.
   \end{cases} 
\]
An induction argument shows that $c$ is a feasible solution of problem \eqref{eq:dualprob1} (satisfying the $T$ first constraints with an equality), so that
\begin{align*}
\frac{f(S_{gr})}{f(S_{opt})} \ge \sum_{i=1}^T c_i. 
\end{align*}
Moreover, it is easy to compute that
\begin{align*}
\sum_{i=1}^{T-m} c_i &= \frac{1}{\alpha}\left(1-\frac{\alpha m}{T} \right)\left(1-\left(1-\frac{\alpha}{T}\right)^{T-m}\right)
\end{align*}
and
\begin{align*}
\sum_{i=T-m+1}^{T} c_i &= \frac{m}{T},
\end{align*}
which, after summation, yield the desired performance guarantee
\begin{align*}
\frac{f(S_{gr})}{f(S_{opt})} \ge G(T,\alpha,m).
\end{align*}

We next show the tightness of $\widetilde G(T,\alpha,n)$ by modifying the proof of \cite[Theorem 5.4]{CC84}. 
Let $1 \le T< n$ be any positive numbers. 
Pick any number $1\le r \le n/2$, let $X=\{a_1,\dots,a_{r}, b_1,\dots, b_{n-r}\}$ and let $f\colon 2^X\to \mathbb R_+$ be the set function
\begin{align*}
f(\{a_{i_1},\dots,a_{i_s},b_{j_1},\dots, b_{j_u}\})=
u+\left(1-\frac{\alpha u}{T}\right)\sum_{k=1}^s\left(1-\frac{\alpha}{T}\right)^{i_k-1},
\end{align*}
defined for all subsets $\{a_{i_1},\dots,a_{i_s},b_{j_1},\dots, b_{j_u}\}\subset X$.
Then $f(\emptyset)=0$. 
For any $S=\{a_{i_1},\dots,a_{i_s},b_{j_1},\dots, b_{j_u}\}\subsetneq X$, where $s<r$ and $u\le n-r$, and
$a_i\in X\setminus S$, we have
\begin{align*}
f(S\cup\{a_i\})-f(S)=
\left(1-\frac{\alpha u}{T}\right)\left(1-\frac{\alpha}{T}\right)^{i-1}\ge 0.
\end{align*}
For any $S=\{a_{i_1},\dots,a_{i_s},b_{j_1},\dots, b_{j_u}\}\subsetneq X$, where $s\le r$ and $u< n-r$, and
$b_j\in X\setminus S$, we have
\begin{align*}
f(S\cup\{b_j\})-f(S)=
1-\frac{\alpha}{T}\sum_{k=1}^s\left(1-\frac{\alpha}{T}\right)^{i_k-1} \ge 0.
\end{align*}
By recalling that a set function $g\colon 2^X\to \mathbb R_+$ is submodular if and only if 
\begin{align*}
g(S\cup\{x\})-g(S)\ge g(R\cup\{x\})-g(R),
\end{align*}
for all $S\subset R\subsetneq X$ and $x\in X\setminus R$ (e.g., \cite{NWF78}), 
these inequalities show that $f$ is submodular and nondecreasing. Moreover, 
\begin{align*}
\max\left\{1-\frac{f(X)-f(X\setminus\{x\})}{f(\{x\})}\colon x\in X, f(\{x\})\ne 0\right\}\\
=1-\frac{f(X)-f(X\setminus\{a_i\})}{f(\{a_i\})}
=\alpha,
\end{align*}
for any $1\le i\le r$, so $f$ has total curvature $\alpha$.

Consider next the case where $T> n/2$. Set $r=n-T$, so that $r<n/2<T$ and $n-r=T$. 
It is easy to verify that $S_{opt}=\{b_{1},\dots,b_{T}\}$ is an optimal solution to problem \eqref{prob1} with $f(S_{opt})=T$.
Since $f(\{a_1\})=f(\{b_j\})=1$, for any $1\le j \le T$, the greedy algorithm can choose the element $a_1$ at the first iteration. Assume next that the greedy algorithm has chosen $S_{t-1}=\{a_1,\dots, a_{t-1}\}$ for some $t \le n-T$. Using the fact 
\begin{align*}
\sum_{k=1}^l
\left(1-\frac{\alpha}{T}\right)^{k-1}=\frac{T}{\alpha}\left(1-\left(1-\frac{\alpha}{T}\right)^l\right)
\end{align*}
it is easy to see that 
\begin{align*}
f(S_{t-1}\cup\{a_t\})=
f(S_{t-1}\cup\{b_j\})=
\sum_{i=1}^t\left(1-\frac{\alpha}{T}\right)^{i-1},
\end{align*}
so the greedy algorithm can choose $a_t$ at the $t$th iteration. 
We therefore can have $S_{gr}=\{a_1,\dots a_{n-T}, b_1,\dots, b_{2T-n}\}$.
This solution has the value
\begin{align*}
f(S_{gr})
=
\frac{T}{\alpha}\left(1-\left(1-\frac{\alpha m}{T}\right)\left(1-\frac{\alpha}{T}\right)^{n-T}\right).
\end{align*}
The claim follows because $m=|S_{opt}\cap S_{gr}|=2T-n$, whence we obtain $n-T=T-m$.

The proof of case $T\le n/2$ is easier, so we omit its proof.

\end{document}